\documentclass{amsart}
\usepackage{amsmath,amssymb,amsthm,url}
\newtheorem{thr}{Theorem}
\newtheorem{lem}[thr]{Lemma}
\newtheorem{conj}[thr]{Conjecture}

\newtheorem{prop}[thr]{Proposition}
\newtheorem{obs}[thr]{Observation}

\theoremstyle{definition}

\newtheorem{pr}[thr]{Problem}

\theoremstyle{remark}

\numberwithin{equation}{section}

\def\A{\mathcal{A}}
\def\N{\mathcal{N}}
\def\U{\mathcal{U}}
\def\V{\mathcal{V}}
\def\C{\mathbb{C}}

\def\I{\mathcal{I}}
\def\T{\mathcal{T}}
\def\S{\mathcal{S}}
\def\L{\mathcal{L}}
\def\E{\mathcal{E}}
\def\M{\mathcal{M}}
\def\B{\mathcal{B}}
\def\D{\mathcal{D}}
\def\rank{\operatorname{rank}}

\begin{document}

\title{A counterexample to Comon's conjecture}

\author{Yaroslav Shitov}
\email{yaroslav-shitov@yandex.ru}






\begin{abstract}
The rank and symmetric rank of a symmetric tensor may differ.
\end{abstract}

\maketitle

We work with three-dimensional \textit{tensors} with complex entries. The entries of such a tensor are indexed by triples $(i,j,k)$ in the Cartesian product of finite \textit{indexing sets} $I, J, K$. A tensor $T$ is called \textit{simple} if $T(i|j|k)=a_ib_jc_k$, for some vectors $a\in\mathbb{C}^I$, $b\in\mathbb{C}^J$, $c\in\mathbb{C}^K$, and \textit{symmetric} if $I=J=K$ and the value of $T(i|j|k)$ does not change under permutations of $(i,j,k)$. The \textit{rank} of $T$ is the smallest $r$ for which $T$ can be written as a sum of $r$ simple tensors, and such a representation is called a \textit{rank decomposition} of $T$. If $T$ is a symmetric tensor, and if simple tensors in decompositions are required to be symmetric, then the corresponding smallest value of $r$ is called the \textit{symmetric rank} of $T$.
A basic result of linear algebra says that these two notions of rank agree for symmetric matrices, and an analogous statement for tensors has become known as \textit{Comon's conjecture}.

\begin{conj}\label{concomon}
The symmetric rank of a symmetric tensor equals its rank.
\end{conj}

Symmetric tensors arise naturally in different applications, so it is an important problem to compute the symmetric ranks and corresponding decompositions, see~\cite{BGI, CGLM, DM, LandBook} and references therein. Conjecture~\ref{concomon} received a considerable amount of attention in recent publications, which include~\cite{BB, BGL, BL, Com1, CGLM, Fried, HL, LandBook, LM, ZHQ}, but it has been proved in several special cases only. The general form of Conjecture~\ref{concomon} remained open, and the aim of this paper is to construct a counterexample.

\section{Setting up}

In this section, we collect several additional definitions, simple observations, known results, and conjectures that arisen in the course of this work. Statements similar to the results of this section may appear in previous work, 
but our goal is to put them in context of what we need below.
For any tensor $T$ in $\C^{I\times J\times K}$, we define the $k$th \textit{$3$-slice} of $T$ as a matrix in $\C^{I\times J}$ whose $(i,j)$ entry equals $T(i|j|k)$. For all $i\in I$, $j\in J$, we can define the $i$th $1$-slice and $j$th $2$-slice of $T$ in a similar way.
If $I_0\subset I$, $J_0\subset J$, $K_0\subset K$, then we denote by $T(I_0|J_0|K_0)$ the \textit{restriction} of $T$ to $I_0\times J_0\times K_0$.
The \textit{support} of $T$ is the smallest set $I_0\times J_0\times K_0\subset I\times J\times K$ containing all the non-zero entries of $T$, and the sets $I_0$, $J_0$, $K_0$ are called the $1$-, $2$-, and $3$-supports of $T$. Two tensors are called \textit{equivalent} if they become equal when restricted to their supports. All the notions introduced above can be defined for matrices in an analogous way.

\subsection{Elementary transformations.}
Let $T=T_1+\ldots+T_n$ be a rank decomposition of a tensor $T$. Let us denote by $M_u$ any non-zero $3$-slice of $T_u$ and find scalars $\alpha_{uk}$ such that the $k$th $3$-slice of $T_u$ is $\alpha_{uk} M_u$.

Assume that a linear combination $M_0=\lambda_1 M_1+\ldots+\lambda_n M_n$ is a rank-one matrix. In this case, one of the $\lambda$'s is non-zero (and to be definite, we assume $\lambda_n\neq0$). We can write
$T=T'_0+T'_1+\ldots+T'_{n-1}$, where $T'_0$ is the tensor whose $k$th $3$-slice equals
$(\alpha_{nk}/\lambda_n) M_0$, and $T'_u$ has $k$th $3$-slice equal to
$(\alpha_{uk}-\lambda_u \alpha_{nk}/\lambda_n) M_u$.

We say that $(T_1,\ldots,T_n)$ and $(T_0',\ldots,T_{n-1}')$ can be obtained from each other by an \textit{elementary $3$-transformation}. 
The same definitions but for $M_u$'s being $1$- or $2$-slices correspond to \textit{elementary $1$- and $2$-transformations}, respectively. 

\begin{obs}\label{lemeqsl}
Let $T=T_1+\ldots+T_r$ be a rank decomposition of a tensor $T$ whose $3$-slices with indexes $\alpha,\beta$ coincide. Then the $3$-slices of every $T_i$ with indexes $\alpha,\beta$ coincide as well.
\end{obs}

\begin{proof}
If the statement was false, we would be able to span the $3$-slices of $T$ by less than $r$ rank-one matrices, which means that the rank of $T$ would be less then $r$.
\end{proof}

\begin{lem}\label{lemtran1}
Let $T\in\C^{I\times J\times K}$ and $K=K'\cup\{1,\ldots,k\}$. Let $S_i$ denote the $i$th $3$-slice of $T$, and assume that $S_{1},\ldots,S_{k}$ are rank-one and linearly independent. Then, for any rank decomposition $(T_1,\ldots,T_n)$ of $T$, there is a sequence of elementary $3$-transformations that sends it into a decomposition $(R_1,\ldots,R_n)$ such that, for all $q\in \{1,\ldots, k\}$, the $q$th $3$-slice of $R_q$ is collinear to $S_q$, and in any such decomposition the $q$th $3$-slice of every $R_u$ with $u\neq q$ is zero.
\end{lem}

\begin{proof}
Since every $3$-slice of $T$ can be expressed as a linear combination of the $3$-slices of $T_i$'s, and since $S_q$ is rank-one, we can find a $3$-transformation resulting in a decomposition $(R_1,\ldots,R_n)$ in which  the $q$th $3$-slices of $R_q$ are collinear to $S_q$. The second assertion of the lemma follows now from Observation~\ref{lemeqsl}.
\end{proof}

\subsection{Eliminating rank-one slices.} 
Assume $\mathcal{T}\subset\C^{I\times J\times K}$ and let $V$ be a $\mathbb{C}$-linear subspace of the matrix space $\C^{I\times J}$. By $\mathcal{T}\,\mathrm{mod}_3 V$ we denote the set of all tensors that can be obtained from some $\tau\in\mathcal{T}$ by replacing every of the $3$-slices $\tau_i$ of $\tau$ with $\tau_i-v_i$, where $v_i$ are matrices in $V$. If $I=J=K$, then we write
$$\E_V(T)=\left((T\,\mathrm{mod}_3 V)\,\mathrm{mod}_2 V\right)\,\mathrm{mod}_1 V.$$
If $S$ is a subset of some $\C$-linear space, then $\operatorname{span}(S)$ denotes the linear span of $S$. The following statement is well known in the community, see e.g. Lemma~2 in~\cite{HK} and Proposition~3.1 in~\cite{LMat}.

\begin{lem}\label{lemcompl}
Let $T\in\C^{I\times J\times K}$ and $K=\{1,\ldots,k\}\cup\{1',\ldots,k'\}$. Let $S_i$ denote the $i$th $3$-slice of $T$, and assume that $S_{1'},\ldots,S_{k'}$ are linearly independent. Then
$$\operatorname{rank} T\geqslant k'+\min \operatorname{rank} T\,\mathrm{mod}_3\operatorname{span}(S_{1'},\ldots,S_{k'}),$$ and if $S_{1'},\ldots,S_{k'}$ are also rank-one, then the equality holds.
\end{lem}


\subsection{Adjoining slices to tensors}

Let $T\in\C^{I\times J\times K}$ be a tensor, and let $$\M_1\subset\C^{J\times K},\,\,\,\, \M_2\subset\C^{I\times K},\,\,\,\,\M_3\subset\C^{I\times J}$$ be finite matrix sets.
We set $I'=I\cup\M_1$, $J'=J\cup\M_2$, $K'=K\cup\M_3$, and we define the new tensor $\mathcal{T}\in\C^{I'\times J'\times K'}$ as follows:

\noindent (1) $\mathcal{T}(\alpha|\beta|\gamma)=T(\alpha|\beta|\gamma)$ if $(\alpha,\beta,\gamma)\in I\times J\times K$;

\noindent (2) for any $\chi\in\{1,2,3\}$ and any $m\in\M_\chi$, the $m$th $\chi$-slice of $\mathcal{T}$ is equivalent to $m$ (that is, coincides with $m$ up to adding zero rows and columns).


We will say that $\mathcal{T}$ is obtained from $T$ by \textit{adjoining} the $1$-slices $\M_1$, the $2$-slices $\M_2$, and the $3$-slices $\M_3$. The result below follows from Lemma~\ref{lemcompl}.

\begin{lem}\label{lemcompl5}
Let $T$, $\mathcal{T}$, $\M_1$, $\M_2$, $\M_3$ are as defined in this subsection. Then
$$\operatorname{rank}\mathcal{T}\geqslant \min \operatorname{rank} \left((T\,\mathrm{mod}_3 V_3)\,\mathrm{mod}_2 V_2\right)\,\mathrm{mod}_1 V_1+\dim V_1+\dim V_2+\dim V_3,$$ where $V_i$ is the $\C$-linear span of $\M_i$. If the matrices in the $\M_i$'s are rank-one, then the equality holds.
\end{lem}

Now assume that $I=J=K$, $\M_1=\M_2=\M_3$, and the tensor $T$ and matrices in $\M_1$ are symmetric. The tensor $\T$ as above is then said to be obtained from $T$ by the \textit{symmetrical adjoining} of matrices in $\M_1$. Unfortunately, the symmetric counterpart of Lemma~\ref{lemcompl5} is just a conjecture.

\begin{conj}\label{conwar}
Let $T\in\C^{I\times I\times I}$ be a symmetric tensor, and let $\M\in\C^{I\times I}$ be a set of symmetric rank-one matrices. Assume $\mathcal{T}$ is obtained from $T$ by the symmetrical adjoining of matrices in $\M$. Then the symmetric rank of $\mathcal{T}$ equals $3\dim\operatorname{span}\M$ plus the minimal symmetric rank of a symmetric tensor in $\E_{\operatorname{span}\M}(T)$.
\end{conj}

At the beginning of the work on this project, the author supposed that at least a weaker form of Conjecture~\ref{conwar} is true. However, he did not manage to prove this conjecture even in the special case when $\M$ is a single matrix. A potential symmetric analogue of Lemma~\ref{lemcompl5} can also be formulated in terms of the \textit{Waring rank} function of homogeneous polynomials. Namely, the quantity $\operatorname{WR}(f)$ is the smallest $w$ for which a polynomial $f$ is the sum of $w$ powers of linear forms.

\begin{conj}\label{conwar2}
Let $f,g\in\C[x_1,\ldots,x_n]$ be homogeneous polynomials of degrees $d$, $d-1$, respectively. Let $u$ be a new variable, and let $\ell[x_1,\ldots,x_n]$ be the set of all linear forms in $x_1,\ldots,x_n$. Then $$\operatorname{WR}(f+ug)\geqslant d+\min_{v\in\ell[x_1,\ldots,x_n]}\operatorname{WR}(f+vg),$$ and if $g$ is a power of a linear form, then the equality holds.
\end{conj}

The author is not aware of counterexamples to Conjectures~\ref{conwar} and~\ref{conwar2}, and he would like to thank Mateusz Miha\l{}ek and Emanuele Ventura for discussing these conjectures with him. Should we be able to prove one of these conjectures, we would present a much simpler counterexample to Comon's conjecture than the one constructed here. The present approach requires a lot of combinatorial issues and technical details to be dealt with, but I hope that the idea of my construction is not completely hidden behind them and can be helpful in studying related problems.



\section{The matrices $\U$, $\V$, $\Lambda$}

The following combinatorial construction is a part of our counterexample.

\begin{obs}\label{obspart}
Consider the following partition
$$\newcommand*{\tempo}{\multicolumn{1}{c|}{1}}
\newcommand*{\temptw}{\multicolumn{1}{c|}{2}}
\newcommand*{\tempth}{\multicolumn{1}{c|}{3}}
\newcommand*{\tempfo}{\multicolumn{1}{c|}{4}}
\newcommand*{\tempfi}{\multicolumn{1}{c|}{5}}
\newcommand*{\tempsi}{\multicolumn{1}{c|}{6}}
\newcommand*{\tempse}{\multicolumn{1}{|c|}{7}}
\newcommand*{\tempei}{\multicolumn{1}{c|}{8}}
\newcommand*{\tempni}{\multicolumn{1}{c|}{9}}
\newcommand*{\tempte}{\multicolumn{1}{c|}{10}}
\newcommand*{\tempoo}{\multicolumn{1}{c|}{11}}
\newcommand*{\temptwo}{\multicolumn{1}{c|}{12}}
\newcommand*{\temptho}{\multicolumn{1}{c|}{13}}
\newcommand*{\tempfoo}{\multicolumn{1}{c|}{14}}
\newcommand*{\tempfio}{\multicolumn{1}{c|}{15}}
\newcommand*{\tempsio}{\multicolumn{1}{c|}{16}}
\newcommand*{\tempseo}{\multicolumn{1}{c|}{17}}
\newcommand*{\tempeio}{\multicolumn{1}{c|}{18}}
\newcommand*{\tempnio}{\multicolumn{1}{c|}{19}}
\newcommand*{\tempteo}{\multicolumn{1}{|c|}{20}}
\begin{array}{|cccccccccc|}\hline
1&1&\tempo&\tempoo&9&9&\tempni&10&10&10\\\cline{4-7}
1&1&\tempo&2&2&\temptw&\temptwo&10&10&10\\\cline{7-10}
1&1&\tempo&2&2&\temptw&3&3&\tempth&13\\\cline{1-3}\cline{10-10}
4&\tempfo&\tempfoo&2&2&\temptw&3&3&\tempth&4\\\cline{3-6}
4&\tempfo&5&5&\tempfi&\tempfio&3&3&\tempth&4\\\cline{6-9}
4&\tempfo&5&5&\tempfi&6&6&\tempsi&\tempsio&4\\\cline{1-2}\cline{9-10}
\tempse&\tempseo&5&5&\tempfi&6&6&\tempsi&7&7\\\cline{2-5}
\tempse&8&8&\tempei&\tempeio&6&6&\tempsi&7&7\\\cline{5-8}
\tempse&8&8&\tempei&9&9&\tempni&\tempnio&7&7\\\cline{1-1}\cline{8-10}
\tempteo&8&8&\tempei&9&9&\tempni&10&10&10\\\hline
\end{array}
$$
of a $100\times 100$ matrix into $20$ submatrices. (Here, each number stands for a $10\times 10$ block.) The union of $k$ of these submatrices is again a submatrix if and only if $k=1$ or $k=20$.
\end{obs}

Let $L^i$ be the $100\times100$ matrix which has ones at the entries of the $i$th submatrix as in Observation~\ref{obspart} and zeros everywhere else. We denote by $u_i, v_i^\top$ the $100$-vectors which have ones, respectively, at the $1$-, $2$-supports of $L^i$ and zeros everywhere else. (In particular, the $u_i, v_i$ are such that the equality $L^i=u_iv_i^\top$ holds.) Taking $C$ to be the permutation matrix corresponding to $(1\,2\,\ldots\,100)$, that is, $C=E_{12}+E_{23}+\ldots+E_{99,100}+E_{100,1}$, we define the $100\times 100$ matrices $$U^i(a)=C^{-a}u_i u_i^\top C^a,\,\,\,V^i(a)=C^{-a}v_i v_i^\top C^a,\,\,\,L^i(a)=C^{-a}u_iv_i^\top C^a.$$ 
Now let $\U^i_{12},\V^i_{12},\Lambda^i_{12}$ be the $500\times500$ matrices which have zeros everywhere except the principal submatrices whose row and column indexes are in $\I_{12}$, and these submatrices are set to equal (assuming $O$ is a zero $100\times 100$ block)
$$\left(\begin{array}{c|c}U^i(1)&O\\\hline O&O\end{array}\right),\,\,\,
\left(\begin{array}{c|c}O&O\\\hline O&V^i(1)\end{array}\right),\,\,\,
\left(\begin{array}{c|c}U^i(1)&L^i(1)\\\hline L^i(1)^\top&V^i(1)\end{array}\right),$$
respectively. 
We define $\U^i_{13},\V^i_{13},\Lambda^i_{13}$ in the same way but with $\I_{12}$, $U^i(1)$, $V^i(1)$, $L^i(1)$ replaced by $\I_{13}$, $U^i(2)$, $V^i(2)$, $L^i(2)$, respectively. If we use $\I_{23}$, $U^i(3)$, $V^i(3)$, $L^i(3)$ instead of the above quadruples, then we get the matrices $\U^i_{23},\V^i_{23},\Lambda^i_{23}$, and if we use $\I_{45}$, $U^i(4)$, $V^i(4)$, $L^i(4)$, then we get $\U^i_{45},\V^i_{45},\Lambda^i_{45}$. Finally, we define $\U^i_{1234},\V^i_{1234},\Lambda^i_{1234}$ to be the matrices which have zeros everywhere except the principal submatrices corresponding to $\I_{1234}$, and these submatrices are
$$\left(\begin{array}{c|c|c|c}U^i(5)&U^i(5)&U^i(5)&O\\\hline U^i(5)&U^i(5)&U^i(5)&O\\\hline U^i(5)&U^i(5)&U^i(5)&O\\\hline O&O&O&O\end{array}\right),\,\,\,
\left(\begin{array}{c|c|c|c}O&O&O&O\\\hline O&O&O&O\\\hline O&O&O&O\\\hline O&O&O&V^i(5) \end{array}\right),$$
$$\left(\begin{array}{c|c|c|c}U^i(5)&U^i(5)&U^i(5)&L^i(5)\\\hline U^i(5)&U^i(5)&U^i(5)&L^i(5)\\\hline U^i(5)&U^i(5)&U^i(5)&L^i(5)\\\hline L^i(5)^\top&L^i(5)^\top&L^i(5)^\top&V^i(5)\end{array}\right),$$ respectively.
Finally, we denote by $\B_*$ the set $\cup_i\{\U^i_{*}, \V^i_*, \Lambda^i_*\}$, and we write $\B=\B_{12}\cup\B_{13}\cup\B_{23}\cup\B_{45}\cup\B_{1234}$. (Here and in what follows, the symbol $*$ in the subscript stands for one of $12,13,23,45,1234$.)
We denote by $\L_*$ (or $\L$) the $\C$-linear space spanned by matrices in $\B_*$ (or $\B$, respectively). We denote by $\D_*^i$ the set containing $\Lambda_*^i$, $\U^i_{*}$, $\V^i_*$ and all other rank-one matrices in the linear span of these, and we write $\D_*=\D_*^1\cup\ldots\cup\D_*^{20}$. In the following lemma, we prove (a somewhat stronger statement than) the fact that $\D_*$ contains all rank-one matrices that belong to $\L_*$.

\begin{lem}\label{lemdif}
Let $\lambda(b)\in\C$ and assume $B_*=\sum_{b\in\B_*}\lambda(b) b$. If all pairwise row differences of $B_*$ are collinear, and if all pairwise column differences of $B_*$ are collinear, then $B_*$ either is a clone or belongs to $\D_*$.
\end{lem}

\begin{proof}
To be definite, we assume $\B_*=\B_{12}$; the other cases are similar. By our construction, the matrix $B_{12}$ has $300$ zero rows and $300$ zero columns, and every non-zero row and every non-zero column appears ten times. We remove all the zero rows and $180$ non-zero rows from $B_{12}$ in such a way that no pair of equal rows remains. Further, we remove the columns with indexes corresponding to those of removed rows, and we denote the resulting matrix by $B'$. 
We have
$$B'=\left(\begin{array}{cc}
U&L\\
L^\top&V
\end{array}\right),$$
and the indexing sets of $L$ correspond to the $10$ rows and $10$ columns of the partitioned matrix as in Observation~\ref{obspart}. Let us denote by $\alpha_i,\beta_i^\top$ the two vectors in $\{0,1\}^{10}$ such that $\alpha_i\beta_i^\top$ has the support of the $i$th matrix in the partition as in Observation~\ref{obspart}. In this notation, the matrix obtained from $\Lambda_{12}^i$ after the transformation described in the first paragraph is $(\alpha_i|\beta_i)(\alpha_i|\beta_i)^\top$, the matrix obtained from $\U_{12}^i$ is $(\alpha_i|0\ldots0)(\alpha_i|0\ldots0)^\top$, and the matrix obtained from $\V_{12}^i$ is
$(0\ldots0|\beta_i)(0\ldots0|\beta_i)^\top$. We are now ready to complete the proof with the use of the above observations.

In particular, if the off-diagonal entries of $U,V$ are zero, then these matrices should contain at most one non-zero entry. If one of these matrices is non-zero, then we can check that $B_{12}\in\D_{12}^{11}\cup\ldots\cup\D_{12}^{20}$. If both $U$ and $V$ are zero matrices, then all the entries of $L$ should be equal, and in this case $B_{12}$ is a clone.

Finally, it remains to consider the case when $U$ (or $V$) has a non-zero non-diagonal entry. In this case, the non-zero entries of $U$ should be concentrated within a principal $3\times 3$ submatrix, and the same statement is true for $V$. We complete the proof by checking that $B_{12}\in\D_{12}^{1}\cup\ldots\cup\D_{12}^{10}$.
\end{proof}

\begin{lem}\label{lemclone}
For $\lambda(b)\in\C$, the matrix $B_{*}=\sum_{b\in\B_{*}}\lambda(b) b$ is a clone if and only if $$\lambda\left(\Lambda_{*}^1\right)=-\lambda\left(\U_{*}^1\right)=-\lambda\left(\V_{*}^1\right)=\ldots=\lambda\left(\Lambda_{*}^{20}\right)=-\lambda\left(\U_{*}^{20}\right)=-\lambda\left(\V_{*}^{20}\right).$$
\end{lem}

\begin{proof}
The conditions $\lambda\left(\Lambda_{*}^1\right)=\ldots=\lambda\left(\Lambda_{*}^{20}\right)$ follow directly from Observation~\ref{obspart}, and the other equations are necessary to make the diagonal blocks zero.
\end{proof}

\begin{lem}\label{lemrank}
The matrices in $\B$ are linearly independent, that is, $\dim\L=300$.
\end{lem}

\begin{proof}
Assume that $B=B_{12}+B_{23}+B_{13}+B_{45}+B_{1234}$ is a zero matrix, where
$B_{*}=\sum_{b\in\B_{*}}\lambda(b) b$. If $B_{12}$ is not a clone, then there is $j\in\{1,2\}$ and $t$ for which the rows of $B_{12}$ with indexes $(j,10t+1)$, $(j,10t+2)$ are different. The corresponding rows of $b\in\B\setminus\B_{12}$ are equal by construction, so the difference between the rows of $B$ with indexes $(j,10t+1)$, $(j,10t+2)$ is non-zero. This contradiction shows that $B_{12}$ is a clone, and we can check that $B_{23},B_{13},B_{45},B_{1234}$ are clones by a similar argument. 
In particular, the matrices $B_*$ are collinear to the clones of
$$E_{12}+E_{21},\,E_{13}+E_{31},\,E_{23}+E_{32},\,E_{14}+E_{24}+E_{34}+E_{41}+E_{42}+E_{43},\,E_{45}+E_{54}$$
(where $E_{ij}$ denotes a matrix unit). These matrices are linearly independent, so the $B_*$'s are zero, and Lemma~\ref{lemclone} implies that all the $\lambda$'s are zero as well.
\end{proof}

\section{A counterexample to Comon's conjecture}

Let $T$ be a symmetric tensor with indexing set $I$. The \textit{clone} of $T$ is the tensor $\mathcal{T}$ obtained from $T$ by taking $\sigma=100$ copies of every element in the indexing set. In particular, the indexing set of $\mathcal{T}$ is $\I=I\times\Sigma$, where $\Sigma=\{1,\ldots,\sigma\}$, and we have $\mathcal{T}(i_1,i_2|j_1,j_2|k_1,k_2)=T(i_1|j_1|k_1)$ for all $i_1,j_1,k_1\in I$ and $i_2,j_2,k_2\in\Sigma$. Also, we will write $\I_{a_1,\ldots,a_n}$ for $\{a_1,\ldots,a_n\}\times\Sigma$. Clearly, taking a clone of a tensor does not change its rank or symmetric rank. We define the $5\times5\times5$ tensor
$$A=\left(\begingroup
\renewcommand*{\arraystretch}{0.95}
\begingroup
\renewcommand*{\arraycolsep}{3pt}
\begin{array}{ccccc}
0&0&0&0&1\\
0&0&0&0&1\\
0&0&0&0&1\\
0&0&0&1&0\\
1&1&1&0&0
\end{array}
\endgroup\endgroup
\left|
\begingroup
\renewcommand*{\arraystretch}{0.95}
\begingroup
\renewcommand*{\arraycolsep}{3pt}
\begin{array}{ccccc}
0&0&0&0&1\\
0&0&0&0&1\\
0&0&0&0&1\\
0&0&0&2&0\\
1&1&1&0&0
\end{array}
\endgroup\endgroup
\left|
\begingroup
\renewcommand*{\arraystretch}{0.95}
\begingroup
\renewcommand*{\arraycolsep}{3pt}
\begin{array}{ccccc}
0&0&0&0&1\\
0&0&0&0&1\\
0&0&0&0&1\\
0&0&0&0&0\\
1&1&1&0&0
\end{array}
\endgroup\endgroup
\left|
\begingroup
\renewcommand*{\arraystretch}{0.95}
\begingroup
\renewcommand*{\arraycolsep}{3pt}
\begin{array}{ccccc}
0&0&0&1&0\\
0&0&0&2&0\\
0&0&0&0&0\\
1&2&0&0&0\\
0&0&0&0&0
\end{array}\endgroup\endgroup
\left|
\begingroup
\renewcommand*{\arraystretch}{0.95}
\begingroup
\renewcommand*{\arraycolsep}{3pt}
\begin{array}{ccccc}
1&1&1&0&0\\
1&1&1&0&0\\
1&1&1&0&0\\
0&0&0&0&0\\
0&0&0&0&0
\end{array}
\endgroup\endgroup
\right)\right.\right.\right.\right.$$
and denote by $\mathcal{A}$ the \textit{clone} of $A$.

We define $\S$ as the tensor obtained by symmetrically adjoining all the matrices in $\B$ to $\A$. Since the tensor $\mathcal{A}$ and the adjoined matrices are symmetric, so is $\S$. The size of $\S$ is $800\times800\times800$, and our goal is to prove that the rank of $\S$ is strictly less than its symmetric rank. Namely, we are going to show that $\S$ can be decomposed into a sum of $903=3\dim\L+3$ simple tensors but not into a sum of $903$ symmetric simple tensors. We can already prove the first part of this statement.

\begin{prop}\label{thrupper}
We have $\rank\,\S\leqslant 903$.
\end{prop}

\begin{proof}
Let us denote by $\mathcal{W}$ the linear span of the matrices
$$E_{12}+E_{21},\,E_{13}+E_{31},\,E_{23}+E_{32},\,E_{14}+E_{24}+E_{34}+E_{41}+E_{42}+E_{43},\,E_{45}+E_{54}.$$
Since the clones of these are $\sum_{i=1}^{20}\left(\Lambda_*^i-\U_*^i-\V_*^i\right)$, the set of clones of matrices in $\mathcal{W}$ is a subspace of $\mathcal{L}$.

According to Lemma~\ref{lemcompl}, we need to check that $\E_\mathcal{L}(\mathcal{A})$ contains a tensor of rank at most three, so it suffices to provide a rank three tensor in $\E_\mathcal{W}(A)$. It remains to check (our Mathematica file~\cite{filecheck} can be used) that the tensor
$$\left(
\begingroup
\renewcommand*{\arraycolsep}{0pt}
\begin{array}{c}
1\\
1\\
1\\
\frac{-1}{\sqrt{2}}\\
0
\end{array}
\endgroup
\right)\otimes
\left(
\begingroup
\renewcommand*{\arraycolsep}{0pt}
\begin{array}{c}
1\\
1\\
1\\
\frac{-1}{\sqrt{2}}\\
0
\end{array}
\endgroup\right)\otimes
\left(
\begingroup
\renewcommand*{\arraycolsep}{0pt}
\begin{array}{c}
-1\\
1\\
-3\\
0\\
1
\end{array}
\endgroup\right)+
\left(
\begingroup
\renewcommand*{\arraycolsep}{0pt}
\begin{array}{c}
0\\
1\\
-1\\
0\\
1
\end{array}
\endgroup\right)\otimes
\left(
\begingroup
\renewcommand*{\arraycolsep}{0pt}
\begin{array}{c}
1\\
1\\
1\\
-1\\
0
\end{array}
\endgroup\right)\otimes
\left(
\begingroup
\renewcommand*{\arraycolsep}{0pt}
\begin{array}{c}
1\\
1\\
1\\
-1\\
0
\end{array}
\endgroup\right)+
\left(
\begingroup
\renewcommand*{\arraycolsep}{0pt}
\begin{array}{c}
1\\
1\\
1\\
\frac{\mathbb{I}}{\sqrt{3}}\\
0
\end{array}
\endgroup\right)\otimes
\left(
\begingroup
\renewcommand*{\arraycolsep}{0pt}
\begin{array}{c}
1\\
-2\\
4\\
0\\
1
\end{array}
\endgroup\right)\otimes
\left(
\begingroup
\renewcommand*{\arraycolsep}{0pt}
\begin{array}{c}
1\\
1\\
1\\
\frac{\mathbb{I}}{\sqrt{3}}\\
0
\end{array}
\endgroup\right),$$
where $\mathbb{I}$ is the imaginary unit,
does in fact belong to $\E_\mathcal{W}(A)$.
\end{proof}

\section{Symmetric tensors in $\E_\L(\A)$}

The goal of this section is to prove a technical statement which would finalize our argument if we were able to prove Conjecture~\ref{conwar}. In particular, we will see that every symmetric tensor in $\E_\L(\A)$ has rank at least four.

We say that indexes $c_1,c_2\in\{1,\ldots,\sigma\}$ are \textit{far apart} if $c_1-c_2$ belongs to $\{30,31,\ldots,70\}$ modulo $\sigma=100$. This condition guarantees that the $(j,c_1|j,c_2)$ entry is zero for any $j\in\{1,2,3,4,5\}$ and any matrix in $\mathcal{L}$.

\begin{lem}\label{lemccl0}
Let $E\in\E_\L(\A)$. If $j\in\{1,2,3\}$, indexes $a,b$ are far apart, indexes $c_1,c_2,c_3$ are pairwise far apart, and indexes $p,q$ are arbitrary, then

\noindent (1) $E(j,a|j,b|5,c_1)=E(j,a|5,c_1|j,b)=E(5,c_1|j,a|j,b)=1,$

\noindent (2) $E(j,c_1|5,a|5,b)=E(5,a|j,c_1|5,b)=E(5,a|5,b|j,c_1)=0,$

\noindent (3) $E(j,c_1|j,c_2|j,c_3)=E(5,c_1|5,c_2|5,c_3)=0$,

\noindent (4) $E(1,1|4,33|4,66)$, $E(2,1|4,33|4,66)$, $E(3,1|4,33|4,66)$ are pairwise different,

\noindent (5) if $i,j\in\{1,2,3\}$ are different, then $E(5,c_1|i,p|j,q)=E(5,c_1|j,q|i,p)$.
\end{lem}

\begin{proof}
Note that $L(j,a|j,b)=L(5,c_1|j,b)=L(j,a|5,c_1)=L(5,c_1|5,c_2)=0$ for any $L\in\mathcal{L}$. Therefore, the entries as in (1)--(3) are equal to the corresponding entries of $\A$.

Further, we check that $L(4,33|4,66)=0$ and $L(1,1|4,t)=L(2,1|4,t)=L(3,1|4,t)$, and the item (4) follows because $\A(1,1|4,33|4,66)$, $\A(2,1|4,33|4,66)$, $\A(3,1|4,33|4,66)$ are pairwise different.

Finally, we note that $L(5,c_1|i,a)=L(5,c_1|j,b)=0$ and $L(i,a|j,b)=L(j,b|i,a)$, for any $L$ in $\L$, and we derive (5) from this similarly to the above statements.
\end{proof}

\begin{lem}\label{lemccl}
Let $E\in\E_\L(\A)$. Then

\noindent (1) for any set $Q\subset\{1,2,3\}\times\Sigma$ of cardinality sixty, the tensor $E(\I\setminus Q|\I\setminus Q|\I\setminus Q)$ has rank at least three.

If $\rank E\leqslant3$, then

\noindent (2) for any $\chi\in\{1,2,3\}$, $q\in\{1,2,3,5\}$, $a,b\in\Sigma$, the $\chi$-slices with indexes $(q,a)$ and $(q,b)$ coincide. In particular, $E(\I_{1235}|\I_{1235}|\I_{1235})$ is the clone of a tensor $F$;

\noindent (3) we also have, for $i,j\in\{1,2,3\}$,

\noindent (3.1) $F(j|j|j)=F(j|5|5)=F(5|j|5)=F(5|5|j)=F(5|5|5)=0$,

\noindent (3.2) $F(j|j|5)=F(j|5|j)=F(5|j|j)=1$,

\noindent (3.3) $F(i|j|5)$, $F(i|5|j)$, $F(5|i|j)$ are $-1$ or $1$,

\noindent (3.4) the $5$th $1$-, $2$-, $3$-slices of $F$ are rank-one.
\end{lem}

\begin{proof}
Fix an arbitrary $c_1\in\{1,\ldots,\sigma\}$ and define $c_0=c_1+1$. We see that $\sigma=100$ is large enough so we can find $c_2,c_3$ such that the only pair of indexes in the sequence $(c_0,c_1,c_2,c_3)$ that are not far apart is $(c_0,c_1)$. We define 
$$I_1=\I_{123}\cup\{(5,c_0),(5,c_1)\},\,\,\,
J_1=\I_{123}\cup\{(5,c_2)\},\,\,\,
K_1=\I_{123}\cup\{(5,c_3)\}.$$
Consider the tensor $E_1=E(I_1|J_1|K_1)$. By Lemma~\ref{lemccl0} (items~2 and~3), the corresponding slices of $E_1$ with indexes $(5,c_i)$ are adjoined to $E(\I_{123}|\I_{123}|\I_{123})$. According to Lemma~\ref{lemcompl5}, the rank of $E_1$ would remain at least three even if we remove $60$ elements from the $\I_{123}$ part of its indexing set. This implies item (1).

Now we assume $\rank E\leqslant3$. We use Lemma~\ref{lemcompl5} again and conclude that the $1$-slices of $E_1$ with indexes $(5,c_0),(5,c_1)$ are rank-one and collinear. Using Lemma~\ref{lemccl0}(1), we conclude that $E(5,c_0|j,a|j,b)=E(5,c_1|j,a|j,b)=1$ holds for all $j\in\{1,2,3\}$ and all indexes $a,b$ that are far apart, and the conclusion of the previous sentence shows that the latter equalities should hold for all $a,b$. In particular, we see that the $1$-slices of $E_1$ with indexes $(5,c_0),(5,c_1)$ coincide.

The previous paragraph shows that the restriction of the $1$-slice of $E_1$ with index $(5,c_0)$ to $\I_{123}$ has the form
\begin{equation}\label{justmatr}
\left(\begin{array}{ccc}
U&?&?\\
?&U&?\\
?&?&U
\end{array}\right),
\end{equation}
where the $U$'s stand for the $100\times100$ blocks of ones. Since this matrix is rank-one, every of the six ``$?$'' blocks should have all entries equal; more than that, these entries should be $\pm1$ because Lemma~\ref{lemccl0}(5) tells us that the matrix~\eqref{justmatr} is symmetric.
A similar reasoning but applied for $2$- and $3$-slices instead of $1$-slices can show us that $$E(5,c_1|j,a|j,b)=E(j,a|5,c_1|j,b)=E(j,a|j,b|5,c_1)=1,$$ and, for $i,j\in\{1,2,3\}$, we get that the values $E(5,c_1|i,a|j,b)$, $E(i,a|5,c_1|j,b)$, $E(i,a|j,b|5,c_1)$ do not depend on $(a,b)$ and are equal to $-1$ or $1$.

Now we use Lemma~\ref{lemcompl5} and find numbers $x_{\alpha},y_{\beta},z_{\gamma}$ such that  $$E(\alpha|\beta|\gamma)=x_{\alpha}E(5,c_1|\beta|\gamma)+y_{\beta}E(\alpha|5,c_2|\gamma)+z_{\gamma}E(\alpha|\beta|5,c_3),$$ provided that $(\alpha,\beta,\gamma)\in I_1\times J_1\times K_1$. This means that $E(1,a|1,b|1,c)=x_{1,a}+y_{1,b}+z_{1,c}$, so by Lemma~\ref{lemccl0}(3) we have $x_{1,a}+y_{1,b}+z_{1,c}=0$ if $a,b,c$ are pairwise far apart. In particular, we get $x_{1,c_1}+y_{1,c_2}+z_{1,c_3}=0=x_{1,c_0}+y_{1,c_2}+z_{1,c_3},$ or $x_{1,c_0}=x_{1,c_1}$. Since $c_1$ was taken arbitrarily, and since $c_0=c_1-1$, this reasoning allows us to conclude that, for $j\in\{1,2,3\}$, the values $x_{j,a},y_{j,b},z_{j,c}$ do not depend on $a,b,c$.

So we have proved that, for $i\in\{1,2,3,5\}$ and $\chi\in\{1,2,3\}$, the difference of the $1$-slices (or $2$-slices, $3$-slices, respectively) of $E$ with indexes $(i,a)$, $(i,b)$ is zero when restricted to $J_1\times K_1$ (or to $I_1\times K_1$, $I_1\times J_1$, respectively). So if these slices were not equal, we would adjoin their difference to $E$ and get a tensor of rank at least $\rank(E_1)+1\geqslant 4$ by Lemma~\ref{lemcompl5}. However, adjoining a linear combination of existing slices cannot change the rank, so we get a contradiction.
\end{proof}

\begin{lem}\label{thrlowereua}
Let $E\in\E_\L(\A)$, $\rank E\leqslant 3$. 
Then there are distinct $i,j\in\{1,2,3\}$ for which the tensor $E(\I_{ij5}|\I_{ij5}|\I_{ij5})$ is not symmetric.
%
%
\end{lem}

\begin{proof}
We apply Lemma~\ref{lemccl} and conclude that $F(i|j|5)=1$ for some distinct $i,j\in\{1,2,3\}$. We set $I=\{(i,1), (j,1), (5,1), (4,33)\}$, $J=\{(i,1), (j,1), (5,1), (4,66)\}$, $K=\{(i,1), (j,1), (5,1)\}$. If $E(\I_{ij5}|\I_{ij5}|\I_{ij5})$ was a symmetric tensor, $E(I|J|K)$ would equal
$$\left(\begin{array}{cccc}
0&a&1&*\\
a&b&1&*\\
1&1&0&*\\
*&*&*&x
\end{array}
\left|\begin{array}{cccc}
a&b&1&*\\
b&0&1&*\\
1&1&0&*\\
*&*&*&y
\end{array}
\left|\begin{array}{cccc}
1&1&0&*\\
1&1&0&*\\
0&0&0&*\\
*&*&*&*
\end{array}\right.\right.\right),$$
where the rows are indexed by $I$, columns by $J$, slices by $K$, and the $*$'s stand for numbers we do not need to specify. According to Lemma~\ref{lemccl0}(4), we have $x\neq y$. One derives a contradiction by checking that such a tensor cannot have rank less than four. (In order to do this, we prove that the slices are linearly independent, but all rank-one matrices spanned by these slices lie in a codimension-one subspace.)
\end{proof}

\section{On symmetric rank decompositions of $\S$}

As said in Section~1, the inability to prove Conjecture~\ref{conwar} was a serious obstruction we faced in this study. In particular, this conjecture together with Lemma~\ref{thrlowereua} would already imply that $\S$ is a counterexample to Conjecture~\ref{concomon}. Now we change our strategy and work with the particular tensor $\S$ instead of the general case; the readers who can prove Conjecture~\ref{conwar} can skip the rest of this section. What we are going to do here is to assume that $\S$ admits a decomposition
\begin{equation}\label{dec0}
\S=\Psi_1+\ldots+\Psi_{3d+3}
\end{equation}
into a sum of $3d+3=3\dim\L+3=3|\B|+3=903$ symmetric simple tensors and to reach a contradiction from this assumption. In order to do this, we perform elementary transformations to get the initial decomposition~\eqref{dec0} into a simpler form, and we will see that the resulting (non-symmetric) decomposition will give us some important information about the initial one.

First, we perform the $1$-transformations on $(\Psi_1,\ldots,\Psi_{3d+3})$ and get
\begin{equation}\label{dec1}
\S=\sum\limits_{b\in\B}\Psi^1_b+\Psi^1_1+\ldots+\Psi^1_{2d+3}=\sum\limits_{b\in\B}\Psi^1_b+\S^1,
\end{equation}
where $\Psi^1_b$ is a tensor with $b$th $1$-slice equivalent (that is, equal up to adjoining zero columns and rows) to $b$, and any other tensor in the decomposition~\eqref{dec1} has zero $b$th $1$-slice. Similarly, the $2$-transformations on $(\Psi^1_1,\ldots,\Psi^1_{2d+3})$ lead us to
\begin{equation}\label{dec2}
\S^1=\sum\limits_{b\in\B}\Psi^2_b+\Psi^2_1+\ldots+\Psi^2_{d+3}=\sum\limits_{b\in\B}\Psi^2_b+\S^2,
\end{equation}
where $\Psi^2_b$ is a tensor with $b$th $2$-slice equivalent to $b$, and any other tensor in the decomposition~\eqref{dec2} has zero $b$th $2$-slice.
Finally, we perform the $3$-transformations on $(\Psi^2_1,\ldots,\Psi^2_{d+3})$ and get
\begin{equation}\label{dec3}
\S^2=\sum\limits_{b\in\B}\Psi^3_b+\Psi^3_1+\Psi^3_2+\Psi^3_3=\sum\limits_{b\in\B}\Psi^3_b+\S^3,
\end{equation}
where $\Psi^3_b$ is a tensor with $b$th $3$-slice equivalent to $b$, and any other tensor in the decomposition~\eqref{dec3} has zero $b$th $3$-slice. In particular, we get the new decomposition
\begin{equation}\label{dec55}
\S=\Psi^3_1+\Psi^3_2+\Psi^3_3+\sum\limits_{b\in\B}(\Psi^1_b+\Psi^2_b+\Psi^3_b)
\end{equation}
of $\S$ into a sum of $3d+3$ simple tensors. Our further strategy is to reduce~\eqref{dec0} into a decomposition similar to~\eqref{dec55} in which sufficiently many tensors $\Psi^1_b+\Psi^2_b+\Psi^3_b$ are symmetric, and then use Lemmas~\ref{lemccl} and~\ref{thrlowereua} to get a contradiction. Our first step is to analyse the $3$-slices of tensors in~\eqref{dec2}. 
In the following lemma and what follows, we may slightly abuse the notation and denote by $b\in\B$ a matrix that is equivalent to $b$ and has size clear from context. We recall that $\D_*$ is the set $\D_*^1\cup\ldots\cup\D_*^{20}$, where $\D_*^i$ consists of $\Lambda_*^i$, $\U^i_{*}$, $\V^i_*$, and all other rank-one matrices spanned by these.

\begin{lem}\label{lemr3}
Denote an (arbitrary) non-zero $3$-slice of the tensor $\Psi_t^3$ as above by $\psi_t^3$. Assume that, for some scalar $\lambda$'s, the matrix
$$\varphi=\sum\limits_{b\in\B}\lambda(b) b+\lambda_1\psi_1^3+\lambda_2\psi_2^3+\lambda_3\psi_3^3$$
is rank-one. Then either $\varphi\in\D_{12}$ or
$$\lambda\left(\Lambda_{12}^1\right)=-\lambda\left(\U_{12}^1\right)=-\lambda\left(\V_{12}^1\right)=\ldots=\lambda\left(\Lambda_{12}^{20}\right)=-\lambda\left(\U_{12}^{20}\right)=-\lambda\left(\V_{12}^{20}\right).$$
\end{lem}

\begin{proof}
By Lemma~\ref{lemclone}, the matrix $B_{12}=\sum_{b\in\B_{12}}\lambda(b) b$ can be a clone only if the displayed condition is satisfied. Otherwise, there is $j\in\{1,2\}$ and $t$ for which the rows of $B_{12}$ with indexes $(j,10t+1)$, $(j,10t+2)$ are different.
However, the corresponding rows of $b\notin\B_{12}$ are equal by the construction, and the corresponding rows of $\psi_i^3$ are equal because of Observation~\ref{lemeqsl} and Lemma~\ref{lemccl}(2).
Therefore, the difference between the rows of $\varphi$ with indexes $(j,10t+1)$, $(j,10t+2)$ equals the difference of the corresponding rows of $B_{12}$, and the rows of $\varphi$ are collinear to this difference. Similarly, the difference between any pair of rows of $B_{12}$ is collinear to the rows of $\varphi$, and we can also derive the same statement for columns instead of rows. The rest follows from Lemma~\ref{lemdif}.
\end{proof}

In what follows, we denote by $\Omega_{12}$ the set of all $b\in\mathcal{D}_{12}$ such that there is $t(b)\in\{1,\ldots,d+3\}$ for which the $3$-slices of $\Psi_{t(b)}^2$ are collinear to $b$.

\begin{lem}\label{lemdec2}
The cardinality of $\Omega_{12}$ is at least $\dim\L_{12}-1=59$.
\end{lem}

\begin{proof}
Since~\eqref{dec2} can be obtained from~\eqref{dec3} by a sequence of elementary $3$-transformations, the $3$-slices of $\Psi_{t}^2$ can only be of the form $\varphi$ as in Lemma~\ref{lemr3}. This shows that the slices outside $\D_{12}$ can span a subspace of dimension at most $4\cdot 60+3+1=244$, but, according to Lemma~\ref{lemrank}, the rank of $\S^2$ cannot be less than $300+3$. Therefore, those $3$-slices of $\Psi_{t}^2$ that lie in $\D_{12}$ should span a subspace of dimension at least $303-244=59$.
\end{proof}

In Lemmas~\ref{lemr1111} and~\ref{lemr111}, we will consider the tensor obtained from $\S$ by the symmetrical adjoining of the matrices in $\Omega_{12}$. Since these matrices are linear combinations of the existing slices of $\S$, the rank decompositions~\eqref{dec0}--\eqref{dec55} can be uniquely extended to the decompositions of the new tensor. Slightly abusing the notation, we preserve the notation used in~\eqref{dec0}--\eqref{dec55} for the new tensor and corresponding decompositions.

\begin{lem}\label{lemr1111}
For any $b\in\Omega_{12}$, there are indexes $t_1(b), t_2(b)\in\{1,\ldots,2d+3\}$ such every tensor $\Psi_{\tau}^1$ has zero $b$th $2$- and $3$-slices except when $\tau\in\{t_1,t_2\}$. 
\end{lem}

\begin{proof}
By the definition of $\Omega_{12}$, the $3$-slices of $\Psi_{t(b)}^2$ are collinear to $b$, which implies that all the other tensors in the decomposition~\eqref{dec2} have zero $b$th $3$-slice. We denote by $u_b$ a non-zero vector collinear to the $3$-slices of the $2$-slices of $\Psi_{t(b)}^2$, that is, collinear to the rows of $b$.

We recall that~\eqref{dec1} can be obtained from~\eqref{dec2} by a sequence of elementary $2$-transformations, and we consider such a transformation involving $\Psi_{t(b)}^2$. In other words, we want to describe rank-one matrices of the form
$$\varphi=\sum\limits_{\beta\in\B}\lambda(\beta) \beta+\sum\limits_{\tau=1}^{d+3}\lambda_\tau\psi_\tau^2,$$
where $\psi_t^2$ denotes a non-zero $2$-slice of the tensor $\Psi_t^2$, the $\lambda$'s are complex numbers, and $\lambda_{t(b)}\neq0$. The above paragraph shows that the $b$th $3$-slice of $\varphi$ is a nonzero vector collinear to $u_b$, so every $3$-slice of $\varphi$ should be collinear to $u_b$.

Further, we need to have $\lambda_{t(v)}=0$ for any $v\in\Omega_{12}\setminus\{b\}$ because otherwise the $v$th $3$-slice of $\varphi$ is non-zero and collinear to $u_v$, which is in turn not collinear to $u_b$. If $\tau\notin t(\Omega_{12})$, then, according to Lemma~\ref{lemr3}, the matrices $\psi_\tau^2$ have $3$-slices collinear to vectors with equal numbers at the positions $(1, 10x+1)$ and $(1,10x+2)$. This is not the case for $u_b$, and since the $\beta$'s have zeros at the positions outside $\I$, the matrix $\mu=\varphi-\lambda_{t(b)}\psi_{t(b)}^2$ should have zero $3$-slices with indexes outside $\I$.

In other words, $\mu$ should be a matrix whose $3$-slices are collinear to $u_b$, and those entries of $\mu$ which have indexes outside $\I$ are zero. Since $\mu$ is a linear combination of the $2$-slices of tensors in rank decomposition, we can adjoin $\mu$ to $\S$ without changing its rank. However, if $\mu$ was not collinear to $b$, then $\mu$ would not be symmetric, and we would have $\mu\notin\L$, which would imply by Lemma~\ref{lemcompl5} that there is a tensor
$E\in\S\,\mathrm{mod}_3 \L\,\mathrm{mod}_1 \L\,\mathrm{mod}_2\mathrm{span}(\L,\mu)$ 
of rank at most $2$. However, the restriction of $E$ to the set $\I\setminus\operatorname{support}(u_b)$ coincides with such a restriction of a tensor in $\S\,\mathrm{mod}_3 \L\,\mathrm{mod}_1 \L\,\mathrm{mod}_2\L$, which, since $u_b$ has either $20$ or $60$ non-zero entries, has rank at least three by the item (1) of Lemma~\ref{lemccl}. The contradiction we have reached shows that $\mu$ is in fact a multiple of $b$.

In other words, $\varphi$ is a linear combination of $b$ and $\psi_{t(b)}^2$, so there can be (at most) two tensors $\Psi_{t_1(b)}^1$, $\Psi_{t_2(b)}^1$ in~\eqref{dec1} having non-zero $b$th $3$-slices. Since~\eqref{dec1} is obtained from a symmetric decomposition~\eqref{dec0} by elementary $1$-transformations, the $1$-slices of $\Psi_{\tau}^1$ should be symmetric, and we see that $\Psi_{t_1(b)}^1$, $\Psi_{t_2(b)}^1$ are still the only tensors in~\eqref{dec1} having non-zero $b$th $2$-slices. 
\end{proof}

\begin{lem}\label{lemr11111}
Assume $V\subset\C^n$ is a two-dimensional linear space containing a vector with non-zero first coordinate. Then the set of all rank-one symmetric matrices whose first row and column are in $V\setminus\{0\}$ spans a three-dimensional linear space.
\end{lem}

\begin{proof}
Let $V$ be spanned by $a = (1, a_2,\ldots,a_n)$ and $b = (1, b_2,\ldots,b_n)$ and, if $x+y\neq0$, we denote by $M_{x,y}$ the rank-one matrix whose first row and column are equal to $xa + yb$. We have $M_{x,y}(i|j)=(xa_i + yb_i) (xa_j + yb_j)/(x+y)$, and it remains to check that $(x+y)M_{x,y}=(x^2-xy)M_{1,0}+(y^2-xy)M_{0,1}+2xyM_{1,1}$.
\end{proof}

\begin{lem}\label{lemr111}
For any $b\in\Omega_{12}$, there are $\tau_1, \tau_2, \tau_3\in\{1,\ldots,3d+3\}$ and simple tensors $\Phi_{b}^\chi$ with $b$th $\chi$-slices equivalent to $b$ such that $\Psi_{\tau_1}+\Psi_{\tau_2}+\Psi_{\tau_3}=\Phi_{b}^1+\Phi_{b}^2+\Phi_{b}^3$.
\end{lem}

\begin{proof}
In the proof of Lemma~\ref{lemr1111} we noticed that the $1$-slices of the tensors $\Psi_{\tau}^1$ are symmetric, so we can apply Lemma~\ref{lemr11111} to the $1$-slices of tensors $\Psi_{t_1(b)}^1$, $\Psi_{t_2(b)}^1$ as in Lemma~\ref{lemr1111}. We see that (at most) three tensors among $\left(\Psi_{\tau}\right)$ can have a non-zero $(b|b|b)$ entry, and the sum of these tensors has the same $b$th slices as $\S$.
\end{proof}

The considerations of Lemmas~\ref{lemr3}--\ref{lemr111} can be repeated for $\B_{13}$ and $\B_{23}$ instead of $\B_{12}$, which leads us to the following.

\begin{lem}\label{lem18}
There are linearly independent subsets $\Omega_{12}\subset\mathcal{D}_{12}$, $\Omega_{13}\subset\mathcal{D}_{13}$, $\Omega_{23}\subset\mathcal{D}_{23}$ of cardinality $59$ each such that, for any $b\in\Omega_{12}\cup\Omega_{13}\cup\Omega_{23}$, there are $\tau_1, \tau_2, \tau_3\in\{1,\ldots,3d+3\}$ and simple tensors $\Phi_{b}^\chi$ with $\chi$-slices collinear to $b$ such that $\Psi_{\tau_1}+\Psi_{\tau_2}+\Psi_{\tau_3}=\Phi_{b}^1+\Phi_{b}^2+\Phi_{b}^3$.
\end{lem}

Let $\beta_{12}\in\D_{12}$, $\beta_{13}\in\D_{13}$, $\beta_{23}\in\D_{23}$ be matrices such that $\Omega_{12}\cup\{\beta_{12}\}$ is a basis of $\L_{12}$, $\Omega_{13}\cup\{\beta_{13}\}$ is a basis of $\L_{13}$, $\Omega_{23}\cup\{\beta_{23}\}$ is a basis of $\L_{23}$. We write $\B'=\B_{45}\cup\B_{1234}\cup\{\beta_{12},\beta_{13},\beta_{23}\}$ and
$$\Phi=\sum\limits_{b\in\Omega_{12}\cup\Omega_{13}\cup\Omega_{23}}(\Phi_b^1+\Phi_b^2+\Phi_b^3).$$
Applying the $1$-, $2$-, $3$-transformations as those in the beginning of this section but with $\S-\Phi$, $\B'$ instead of $\S$, $\B$, we get a decomposition
\begin{equation}\label{dec00}
\S-\Phi=\Xi_1+\Xi_2+\Xi_3+\sum\limits_{b\in\B'}(\Xi_b^1+\Xi_b^2+\Xi_b^3),
\end{equation}
with $\Xi^\chi_b$ having $\chi$-slices collinear to $b$. Let $\N$ be the last summand of~\eqref{dec00}.

\begin{lem}\label{lem00}
If $i,j,k\in\{1,2,5\}$, then there are $c_1,c_2,c_3\in\Sigma$ such that $\N(\pi)=0$ for any permutation $\pi$ of the family $(i,c_1|j,c_2|k,c_3)$.
\end{lem}

\begin{proof}
We note that $L(1,a|5,b)=L(2,a|5,b)=L(i,a|i,b)=0$ for all $L\in\L$ and all $a,b$ that are far apart. Therefore, we can take $(c_1,c_2,c_3)$ to be arbitrary indexes that are pairwise far apart and get the desired result immediately for $(i,j,k)$ equal to $(i,i,i)$, $(i,i,5)$, $(i,5,5)$ or any permutation of these.

Further, we denote by $C_1,C_2\subset\Sigma$ the sets for which $\{1\}\times C_1\cup\{2\}\times C_2$ is the support of $\beta_{12}$. Since neither $C_1$ nor $C_2$ equals $\Sigma$, we can find two families of pairwise far apart indexes $(c_1,c_2,c_3)$, $(c_4,c_5,c_6)$ such that $c_1\notin C_1$ and $c_6\notin C_2$. Then the entries $(1,c_1|2,c_2|2,c_3)$, $(1,c_1|2,c_2|5,c_3)$, $(1,c_4|1,c_5|2,c_6)$, and their permutations are zero in $\N$.
\end{proof}

Now we apply Lemma~\ref{lemccl} with $E$ equal to the restriction of the tensor $\Xi_1+\Xi_2+\Xi_3$ as in~\eqref{dec00} to $\I$, and we are going to discuss the corresponding tensor $F$. For all $i,j,k\in\{1,2,5\}$, the value $F(i|j|k)$ is equal to the $(i,c_1|j,c_2|k,c_3)$ entry of $\S-\Phi-\mathcal{N}$. According to Lemma~\ref{lem00}, and since $\S-\Phi$ is a symmetric tensor, we can choose $c_1,c_2,c_3$ such that $E(\pi)$ is the same for any permutation $\pi$ of $(i,c_1|j,c_2|k,c_3)$. Therefore, $F(1,2,5|1,2,5|1,2,5)$ is a symmetric tensor as well.

Finally, we note that Lemma~\ref{lem00} can be proved with $\{1,3,5\}$ and $\{2,3,5\}$ instead of $\{1,2,5\}$. So we see that $F(1,3,5|1,3,5|1,3,5)$ and $F(2,3,5|2,3,5|2,3,5)$ are also symmetric tensors, which gives a contradiction with Lemma~\ref{thrlowereua}.

\section{Further work}

Several problems related to Conjecture~\ref{concomon} remain open. In particular, we can define the symmetric rank of a tensor $T$ with respect to any field $\mathbb{F}$ as the smallest integer $r$ for which $T$ is an $\mathbb{F}$-linear combination of $r$ symmetric simple tensors with coefficients in $\mathbb{F}$.

\begin{pr}\label{prdf}
Does there exist a field over which the rank and symmetric rank of any symmetric tensor agree? 
\end{pr}

Problem~\ref{prdf} has been considered over $\mathbb{R}$ in~\cite{Fried} but remains open. Joint efforts of the author and Mateusz Micha\l{}ek did not allow to construct a real decomposition of $\S$ of rank $903$, which suggests that the counterexample provided in this paper cannot disprove the real version of Comon's conjecture. However, the technique presented in this paper allows the author to expect a negative answer to Problem~\ref{prdf}.

Another interesting question is the partially symmetric version of Comon's conjecture. A $3$-dimensional tensor is called \textit{partially symmetric} if its $3$-slices are symmetric, and the notion of \textit{partially symmetric rank} arises in the same way as rank but simple tensors in decompositions are required to be partially symmetric.

\begin{pr}\label{prps} (See~\cite{BL}.)
Does there exist a partially symmetric tensor with different rank and partially symmetric rank? 
\end{pr}

This statement seems to be closer than Comon's conjecture to the theorem on the equality of rank and symmetric rank of matrices, and it may be harder to disprove. The author does not know if the present technique is applicable to Problem~\ref{prps}, but he thinks he managed to construct quite a complicated counterexample showing that the solution of Problem~\ref{prps} is negative. The author hopes to revisit the real and partially symmetric versions of Comon's conjecture in future work.

Finally, let us mention the analogue of Comon's conjecture for \textit{border ranks} of tensors, see~\cite{BGL}. Recall that the \textit{(symmetric) border rank} of a real or complex tensor $T$ is the smallest $r$ such that $T$ is the limit of a sequence of (symmetric) tensors with (symmetric) rank at most $r$.

\begin{pr}(See~\cite{BGL}.)\label{prbr}
Does there exist a symmetric tensor with different border rank and symmetric border rank?
\end{pr}

The author did not manage to make any progress on Problem~\ref{prbr}. While preparing the first draft of this paper, he was hoping that the border rank analogues of Lemma~\ref{lemcompl5} and Conjecture~\ref{conwar} could be helpful for the study of this problem. The author would like to thank Mateusz Micha\l{}ek and Emanuele Ventura for explaining to him why do the analogues of these statements fail for border rank.

\section{Acknowledgment}

I am grateful to Mateusz Micha\l{}ek for his interest to the content of this paper, for careful reading of the first and second drafts, for explaining to me several details that I did not understand, which include the behavior of the border rank with respect to adjoining slices as in Section~1, and for numerous comments that allowed me to make this paper more readable.


\begin{thebibliography}{99}

\bibitem{BB}
E. Ballico, A.  Bernardi, Tensor ranks on tangent developable of Segre varieties, \textit{Linear Multilinear A.} 61.7 (2013) 881-894.

\bibitem{BGI}
A. Bernardi, A. Gimigliano, M. Ida, Computing symmetric rank for symmetric tensors, \textit{J. Symb. Comput.} 46.1 (2011) 34-53.

\bibitem{BGL}
J. Buczy\'{n}ski, A. Ginensky, J. M. Landsberg, Determinantal equations for secant varieties and the Eisenbud--Koh--Stillman conjecture, \textit{J. London Math. Soc.} 88.1 (2013) 1-24.

\bibitem{BL}
J. Buczy\'{n}ski, J. M. Landsberg, Ranks of tensors and a generalization of secant varieties, \textit{Linear Algebra Appl.} 438.2 (2013) 668-689.



\bibitem{Com1}
P. Comon, Tensors: a brief introduction, \textit{IEEE Signal Proc. Mag.} 31.3 (2014) 44-53.


\bibitem{CGLM}
P. Comon, G. Golub, L-H. Lim, B. Mourrain, Symmetric tensors and symmetric tensor rank, \textit{SIAM J. Matrix Anal. A.} 30.3 (2008) 1254-1279.

\bibitem{DM}
M. C. Dog\v{a}n, J. M. Mendel, Applications of cumulants to array processing. I. Aperture extension and array calibration. \textit{IEEE Trans. Signal Process.} 43 (1995) 1200-1216.

\bibitem{Fried}
S. Friedland, Remarks on the symmetric rank of symmetric tensors, \textit{SIAM J. Matrix Anal.} 37.1 (2016) 320-337.

\bibitem{HL}
C. J. Hillar , L. H. Lim, Most tensor problems are NP-hard, \textit{J. ACM} 60.6 (2013) 45.

\bibitem{HK}
J. E. Hopcroft, L. R. Kerr, On minimizing the number of multiplications necessary for matrix multiplication, \textit{SIAM J. Appl. Math.} 20 (1971) 30--36.

\bibitem{LandBook}
J. M. Landsberg. \textit{Tensors: geometry and applications.} American Mathematical Society. Providence, RI, USA, 2012.

\bibitem{LMat}
Landsberg J. M., M. Micha\l{}ek, Abelian tensors, preprint (2015). To appear in \textit{J. Math. Pure Appl.}

\bibitem{LM}
J. M. Landsberg, M. Micha\l{}ek, On the Geometry of Border Rank Decompositions for Matrix Multiplication and Other Tensors with Symmetry, \textit{SIAM J. Appl. Algebra Geometry} 1.1 (2017) 2-19.

\bibitem{LT}
J. M. Landsberg, Z. Teitler, On the ranks and border ranks of symmetric tensors, \textit{Found. Comput. Math.} 10.3 (2010) 339-366.


\bibitem{ZHQ}
X. Zhang, Z. H. Huang, L. Qi,  Comon's Conjecture, Rank Decomposition, and Symmetric Rank Decomposition of Symmetric Tensors, \textit{SIAM J. Matrix Anal. A.} 37.4 (2016) 1719-1728.

\bibitem{filecheck}
\url{http://bit.ly/2rThaSv}

\end{thebibliography}
\end{document}